\documentclass[draft]{amsart}
\usepackage{amssymb,latexsym,amscd, amsthm, epsfig,amsmath,amssymb}
\usepackage{color}
\usepackage{array}
\newcolumntype{C}{>{$}c<{$}}
\newcolumntype{L}{>{$}l<{$}}
\newcolumntype{R}{>{$}r<{$}}

 \newtheorem{thm}{Theorem}[section]
 \newtheorem{cor}[thm]{Corollary}
 \newtheorem{lem}[thm]{Lemma}
  
 \newtheorem{prop}[thm]{Proposition}
 \theoremstyle{definition}
 \newtheorem{defn}[thm]{Definition}
 \theoremstyle{remark}
 \newtheorem{rem}[thm]{Remark}
 \theoremstyle{definition}
 \newtheorem{ex}[thm]{Example}

\newtheorem{prob*}{Problem}

 \newcommand{\PP}{\mathbb{P}}

\def\move-in{\parshape=1.75true in 5true in}

\usepackage{xypic}

\begin{document}

\title{{On Parameterizations of plane rational curves and their syzygies.}}

\author[A. Bernardi]{Alessandra Bernardi}
\address[Alessandra Bernardi]{Dipartimento di Matematica, Univ. Bologna,  Italy}
\email{alessandra.bernardi5@unibo.it}


\author[A. Gimigliano]{Alessandro Gimigliano}
\address[Alessandro Gimigliano]{Dipartimento di Matematica e CIRAM,  Univ. Bologna, Italy}
\email{Alessandr.Gimigliano@unibo.it }

\author[M. Id\`a]{Monica Id\`a}
\address[Monica Id\`a]{Dipartimento di Matematica,  Univ. Bologna, Italy}
\email{monica.ida@unibo.it }

\maketitle


\begin{abstract}
 {Let $C$ be a plane rational curve of degree $d$ and $p:\tilde C \rightarrow C$ its normalization. We are interested in the {\it splitting type} $(a,b)$ of $C$, where $\mathcal{O}_{\mathbb{P}^1}(-a-d)\oplus \mathcal{O}_{\mathbb{P}^1}(-b-d)$ gives the syzigies of the ideal $(f_0,f_1,f_2)\subset K[s,t]$, and \, $(f_0,f_1,f_2)$ is a parameterization of $C$. We want  to describe in which cases $(a,b)=(k,d-k)$ ($2k\leq d)$, via a  geometric description; namely we show that $(a,b)=(k,d-k)$ if and only if $C$ is the projection of a rational curve on a rational normal surface in $\PP^{k+1}$.}
\end{abstract}


\section{Introduction}
We work over an algebraically closed ground field $K$ {of characteristic 0}. We are
interested in projective morphisms  $f:\PP^1\to\PP^2$, where $f = (f_0,f_1,f_2)$, $f_i\in K[s,t]_d$,  which are generically injective
and generically smooth over their images; thus obtaining that $f(\PP^1)$ is a rational plane curve $C$ parameterized by $(f_0,f_1,f_2)$.

In particular we are interested in studying the {\it splitting type} of $C$, which is the pair $(k,d-k)$ such that the minimal graded free resolution of the ideal $(f_0,f_1,f_2)$, as a graded $S$-module, $S= K[s,t]$, has $S(-k-d)\oplus S(k-2d)$ as its first syzygy module (see (2) in the next section).

It is easy to see (e.g. see \cite{GHI4}) that the splitting type of $C$ which we have just defined is also the pair  $(k,d-k)$, with $2k\leq d$, such that $f^*(\Omega_{\PP^2}(1)) \cong \mathcal{O}_{\PP ^1}(-k)\oplus \mathcal{O}_{\PP ^1}(-d+k)$.

Actually, it is well-known that any
vector bundle on $\PP^1$ splits as a direct sum of line bundles (see \cite{refB, refG}), and the determination of the splitting of the pull back $f^*T_{\mathbb{P}^n}$ (or, which is equivalent, of $f^*\Omega_{\mathbb{P}^n}(1)$), together with the problem of the splitting of the normal bundle, is a quite classical and  very investigated problem in Algebraic Geometry  for rational curves in $\mathbb{P}^n$; e.g. for $n\geq 3$ and smooth rational curves  see  \cite{EV1, EV2, Ram, A2, Hu, GS, MR3147368, Bernardi:2012qb, MR3296188, Cox:2015:SRS:2735601.2741380}). 

From a more applied point of view, this problem, mainly for plane curves, has been studied in relation to the problem of determining the implicit equations of parameterized plane curves (e.g. see \cite{refCSC}), where the splitting type is studied via the equivalent notion of  {\it $\mu$-basis }, which is defined for the ideal of the graph of $C$, in the affine case; here we will also use that approach, just with the difference that we will work in $\PP^2\times \PP^1$.

Our interest in the splitting type for plane curves have stemmed also by the fact that it is strictly correlated to the problem of finding the minimal free resolution of fat points in the plane (see  \cite{GHI1, GHI2, GHI3, GHI4}).

We are interested in finding  the relation  between  the geometry of the curve $C$ and its splitting type.  One possibility is to study the relation  between the singularities of $C$ and its splitting type.
For example, if $C$ has a point of multiplicity $m$, then
results of Ascenzi \cite{refAs1} show that
\begin{equation}\label{Ascenzibnds}
\min(m,d-m)\leq a\leq\min\Big(d-m,\Big\lfloor \frac{d}{2}\Big\rfloor\Big);
\end{equation}

\noindent see also \cite{GHI4}. 
If $2m+1\geq d$, when $m$ is the maximum multiplicity of the points of $C$, it follows from these bounds that $a=\min(m,d-m)$ and hence $b=\max(m,d-m)$.
So we give the following definition.

\begin{defn}
A rational projective plane curve $C$ is \emph{Ascenzi} if it has a point of multiplicity $m$, with $2m+1\geq d$.
\end{defn}

Notice that the Ascenzi behavior is not the general one, in the sense that it is possible that a curve of  splitting type  $(m,d-m)$ does not have any point of multiplicity $m$. For example the curve  $C$ of degree eight, defined by its syzygy matrix $A$ (see (\ref{eqn1}) below), where:
$$A:=\left( \begin{array}{ccc}
s^3&s^2t+st^2&t^3\\
s^5+3s^2t^3&t^5+3s^3t^2&s^5+t^5+st^4
\end{array}\right),$$
has only double points as singularities while its splitting type is $(3,d-3)$. 

  
The non-Ascenzi cases are more difficult to handle {and only partial results are known} (e.g. see \cite{GHI1, GHI2, GHI4}). 

In \cite{refAs1}, the cases $(a,b)=(1,d-1)$ (any $d$) and $(a,b)=(2,d-2)$, $d\geq 3$, are treated; in the first case $C$ has to be Ascenzi, while in the second either $C$ is Ascenzi or it has only nodes as singularities and it is a projection of a smooth curve $D\subset \mathbb{P}^3$ of degree $d$ which is contained in a quadric surface.
\\
Hence another way to relate the geometry of $C$ with its splitting type is to view $C$ as projection of a curve $D$ lying on  some special surfaces.

In this paper we carry on this perspective by determining that the right surfaces to look at are rational normal scrolls. Notice that in general rational curves can be contained in rational normal scrolls only in particular cases, as the following remark shows.

\begin{rem}
Let us consider  a generic projection $C\subset \PP^m$ of  a rational normal curve $C_d\subset \PP^d$. The curve $C$ has a good postulation, therefore $h^0(\mathcal{I}_C(2))={m+2 \choose 2}-2d-1$ (\cite{Hir, BE}). On the other hand, a rational normal surface $S\subset\PP^m$ verifies $h^0(\mathcal{I}_S(2))={m-1\choose 2}$ (see e.g. \cite{segre1884sulle}). Hence, under a generic projection  $C$ can possibly lie on a rational normal surface only if $m\geq \frac{2d+1}{3}$.
\end{rem}

In this paper  we manage to prove, in Theorem \ref{ex-conjecture}, that for $(a,b)=(k,d-k)$, $d\geq 2k$, the curve $C$ is the projection of a degree $d$ curve $D\subset \mathbb{P}^{k+1}$ which is contained in a rational normal scroll (and we have the Ascenzi case if the rational scroll is a cone and $D$ passes through its vertex with multiplicity $d-k$, otherwis3e $D$ is smooth).
 
Hence we relate the splitting type of the plane rational curve with the existence of a rational curve on a scroll  in a space of dimension ``smaller than expected''.
\\
The main idea in order to construct such a scroll is to define it from the ``moving line" (in the language of \cite{refCSC}) defined  by a syzygy of minimum degree of $C$.

\bigskip

{The paper is organized as follows. In Section \ref{Preliminaries} we give a first result which connects the splitting type of a plane rational curve with its being a projection of a curve on a rational normal scroll. Moreover, for curves of splitting type $(3,d-3)$, we give an explicit parameterization of a curve on a rational normal scroll in $\mathbb{P}^4$ which projects to our plane curve $C$.
In Section \ref{MainThm} we prove our main result, i.e. Theorem \ref{ex-conjecture}.} We end the paper with Section \ref{OpenPrb} where we list a few open problems related to these questions.

\bigskip
\section{Preliminary results}\label{Preliminaries}

We want to study rational plane curves $C$ given via linear systems $\langle f_0,f_1,f_2\rangle \subset
K[s,t]_d$, i.e. $g^2_d$'s on $\PP ^1$ which give a projective morphism, generically one to one,  
 $f : \  \PP ^1 \rightarrow \PP ^2$, whose image is  $C$. We are interested in studying the splitting type of $C$, i.e. the pair $(k,d-k)$ such that the minimal graded free resolution of $J=(f_0,f_1,f_2)$, as a graded $S$-module, $S=K[s,t]$, is: 
\begin{equation}\label{eqn1}
0 \to  S(-k-d)\oplus S(k-2d)
\xrightarrow{\hbox{\tiny $A=\left(
\begin{matrix}
\alpha_0 & \alpha_1 & \alpha_2    \\
\beta _0 & \beta _1 & \beta _2  
\end{matrix}
\right)$}}
S(-d)^{ \oplus3}
\xrightarrow{\hbox{\tiny $(f_0\ f_1\ f_2)$}}
J\to 0,
\end{equation}
where $\deg \alpha_j=k$, $\deg \beta_j=d-k$, and the $f_i$'s are the $2\times 2$ minors of $A$, the Hilbert-Burch matrix of $J$.  For this reason we will speak of  ``splitting type" not only for parameterized rational curves in $\mathbb{P}^2$ but more generally for finite maps $\psi: \mathbb{P}^1 \rightarrow \mathbb{P}^2$, maybe not generically one to one. More precisely we will say that $\psi$ has splitting type $(h,k-h)$ if the ideal $(\psi_0,\psi_1,\psi_2)\subset K[s,t]$ has minimal free resolution:
$$0 \to  S(-h-k)\oplus S(h-2k)
\rightarrow
S(-k)^{ \oplus3}
\xrightarrow{\hbox{\tiny $(\psi_0\ \psi_1\  \psi_2)$}}
J\to 0.
$$

\medskip
Our aim is to relate the splitting $(k,d-k)$ to a geometric property of $C$, namely to how $C$ can be obtained as a projection of a rational curve $D\subset \PP^{k+1}$ which is contained in a rational normal surface $S\subset \PP^{k+1}$ (for such surfaces see e.g. \cite{segre1884sulle, MR927946}).

\medskip
{ The following result describes how curves on rational normal surfaces project to plane curves whose splitting type is related to the scroll itself.}

 \bigskip
\begin{prop}\label{Projection} Let $D\subset \PP^{k+1}$ be a smooth  rational curve of degree $d\geq 2k$ and assume that $D$ is contained in a rational normal surface $S\subset \PP^{k+1}$ and it is a unisecant for the fibers of $S$. Consider points  $P_1,\ldots P_{k-1}\in \PP^{k+1}-S$, such that by projecting $D$ in $\PP^2$ from them, the projection is generically 1:1 onto its image $C$. Then the splitting type of $C$ is $(k',d-k')$, with $k'\leq k$.
\end{prop}

\begin{proof}
The fibers of $S$ define a line bundle $\mathcal{L}_S\subset T_{\PP^{k+1}}\vert _D$; by  \cite[Theorem 1.4]{refAs1},  we have that 
$$
\deg \mathcal{L}_S\ =\ \deg( g_S^*(\omega)) + 2d,
$$
where $\omega$ is  the universal 2-vector bundle on the grassmannian $G = \mathbb{G}(1,k+1)$ and the Gauss map  $g_S: D\rightarrow \mathbb{G}(1,k+1)$ is defined by sending each $P\in D$ to the point of $G$ corresponding to the line of $S$ passing through $P$. 

We get that the image of $g_S$ in $\mathbb{G}(1,k+1)$
is a rational curve $\Gamma$ of degree $k$, in fact $S$ is a rational normal scroll in $\mathbb{P}^{k+1}$, hence $\deg(S)=k$, and also $\deg(\Gamma)=k$ (e.g. see \cite[Section 3]{refCo}).

 Now, since $\deg(\Gamma)=k$ and $g_S$ is an isomorphism between $\mathbb{P}^1$ and its image and $\omega|_{\Gamma}\simeq \mathcal{O}_{\mathbb{P}^1}(-k)$, we have that  $\deg g_S^*(\omega) = -k$ (it is the back image of $\mathcal{O}_G(-1)\vert _\Gamma$). Hence $\mathcal{L}_S$ is a subbundle of $T_{\PP^{k+1}}\vert _D$ of degree $(2d-k)$. 

We get $C$ by a projection $\nu :\PP^{k+1}\rightarrow \PP^2$ from points outside $S$, hence none of the lines belonging to the scroll $S$ contracts via $\nu$ to a point in $\PP^2$. Let $\psi : \PP^1 \rightarrow D\subset \PP^{k+1}$ be a parameterization of $D$ such that $\nu \circ \psi = f$.  The image of the bundle $\mathcal{L}_S$  through the map of vector bundles $\psi^*(T_{\PP^{k+1}}) \to f^*(T_{\PP^2})$ is the line bundle $\mathcal{O}_{\PP^1}(2d-k)$.

Hence we get that  $f^*(T_{\PP^2}) \simeq \mathcal{O}_{\PP^1}(2d-k+\delta)\oplus\mathcal{O}_{\PP^1}(k+d-\delta)$, for some $\delta$,\, $k-1\geq\delta \geq 0$.  Now, dualizing and twisting by $1$
$$
f^*(\Omega_{\PP^2}(1) )\cong \mathcal{O}_{\PP^1}(-k+\delta)\oplus \mathcal{O}_{\PP^1}(-d+k-\delta),$$
i.e. the splitting type of $C$ is $(k',d-k')$, with $k'=k-\delta\leq k$.
\end{proof} 

\bigskip
In \cite{refAs1}, Ascenzi showed that when the splitting type is $(2,d-2)$, $C$ can be viewed as the projection of a curve $D\subset \PP^3$ contained in a quadric surface $S$ from a point outside $S$. The surface   $S$ and the curve $D$ are smooth except when $C$ is Ascenzi, i.e. it possesses a point of multiplicity $d-2$; in that case $S$ is a quadric cone and $D$ has multiplicity $d-2$ at the vertex of $S$ (its only  singularity).

\medskip
 Ascenzi gives a direct construction of the quadric $S$ and of the parameterization of $D$ in $\PP^3$, via  the Hilbert-Burch matrix of $(f_0,f_1,f_2)$. We want to generalize this point of view to any possible splitting type for a reduced, irreducible plane rational curve. 
 
 For curves with splitting type $(3,d-3)$, we are able to strictly follow Ascenzi's idea, in order to construct $S$ and $D$ as in the following
 
 \bigskip
 \begin{prop} Let $C\subset \PP^2$ be a rational curve of degree $d\geq 6$, with splitting type $(3,d-3)$; then there is a rational curve $D$ and a rational normal surface $S$, with $D\subset S \subset \PP^4$, such that $C$ is the projection of $D$ from two points not on $S$. Moreover: 
 \begin{itemize}
 \item if $C$ is not Ascenzi, then $D$ and $S$ are smooth;
 
 \item if $C$ is Ascenzi, then $S$ is a cone, $D$ passes through the vertex of $S$ with multiplicity $(d-3)$ and that one is its only singular point.   
 \end{itemize} 
 \end{prop}
 
 \begin{proof}
If the splitting type of $C$ is $(3,d-3)$, from (\ref{eqn1}) we get that we can view $C$ as parameterized by the maximal minors of  
\begin{equation}\label{syzygies}
A\ =\ \left(\begin{matrix}
\alpha_0 & \alpha_1 & \alpha_2    \\
\beta _0 & \beta _1 & \beta _2  
\end{matrix}\right)
\end{equation}
where $\deg \alpha_j=3$ and $\deg \beta_j=d-3$. Now suppose $C$ is not Ascenzi, i.e. it does not possess a point of multiplicity $d-3$. This implies that its degree 3 syzygy has lenght 3, i.e. the $\alpha_i$'s are linerly independent (see \cite{refAs1}), and, without loss of generality, we can suppose that 
$$
\alpha_0\ =\ s^3,\quad \alpha_1\ =\ s^2t+st^2,\quad    \alpha_2\ =\ t^3
$$
In fact, if $C$ has no points of multiplicity $d-3$, then $\alpha_0, \alpha_1, \alpha_2 $  span a plane $H$ in the space $\mathbb{P}^3$ parameterizing forms of degree 3 in $K[s,t]$. Let $H$ meet in three distinct points the rational normal curve $C_3$ (whose points correspond to forms which are powers of linear forms); via a linear change of coordinates in $\PP^1$, which leaves $C$ unchanged, we can suppose that those three points correspond to $s^3, (s+t)^3$ and $t^3$.  
Since what we need is to keep the same $g^2_d$ defining $C$,  we can change the parametrization $(f_0,\ f_1,\ f_2)$ (i.e. the generators we pick for the linear space $\langle f_0,\ f_1,\ f_2\rangle$) and we can work modulo projectivities  on $\langle f_0,\ f_1,\ f_2\rangle $; in particular if we consider (from Equation \ref{eqn1}):
$$
\left(\begin{matrix}
\alpha_0 & \alpha_1 & \alpha_2    \\
\beta _0 & \beta _1 & \beta _2  
\end{matrix}\right) \cdot \left(\begin{matrix}
f_0   \\
f_1 \\ f_2   
\end{matrix}\right)\ =\ \left(\begin{matrix}
0 \\
0 
\end{matrix}\right),
$$
we can change to:
$$
\left(\begin{matrix}
\alpha_0 & \alpha_1 & \alpha_2    \\
\beta _0 & \beta _1 & \beta _2  
\end{matrix}\right) \cdot P\cdot P^{-1}\cdot \left(\begin{matrix}
f_0   \\
f_1 \\ f_2   
\end{matrix}\right)\ =\ \left(\begin{matrix}
0 \\
0 
\end{matrix}\right),
$$
where $P$ is a $3\times 3$ invertible matrix, and $P^{-1}$ will give a projectivity in the plane  $\langle f_0,\ f_1,\ f_2\rangle$, so we can consider the curve given by parameterization: 
$$P^{-1} \cdot \left(\begin{matrix}
f_0   \\
f_1 \\ f_2   
\end{matrix}\right),\quad {\rm which\ has\ syzygy\ matrix:}\quad \left(\begin{matrix}
\alpha_0 & \alpha_1 & \alpha_2    \\
\beta _0 & \beta _1 & \beta _2  
\end{matrix}\right) \cdot P.$$
 Hence, modulo two projectivities, one in $H$ and one in $\langle f_0,f_1,f_2 \rangle$, given by appropriate $P$ and $P^{-1}$, we can suppose $
\alpha_0\ =\ s^3;\ \alpha_1\ =\ s^2t+st^2;\  \alpha_2\ =\ t^3
$ (since $s^2t+st^2= 1/3[(s+t)^3-s^3-t^3]$).  

\medskip
Notice that if $H$ were to meet $C_3$ only in two points, say $s^3$ and $t^3$, this would imply that the tangent line to $C_3$ at one of them, say $s^3$, is contained in $H$; in this case we can choose $\alpha_1$ on the tangent line of the rational normal  cubic curve at $s^3$, hence we can use: $\alpha_0\ =\ s^3;\ \alpha_1\ =\ s^2t;\  \alpha_2\ =\ t^3$, and all what follows could be done with the same procedure (with even easier computations). 

On the other hand, it is not possible for $H$ to meet $C_3$ at only one point, since in this case $H$ would be the osculating plane of $C_3$ at this point, say $s^3$, and all the forms parameterized by $H$ would have $s$ as a common factor, hence $(\alpha_0,\alpha_1,alpha_2)$ could not be a primitive syzygy.

\bigskip
Now set 
$$\left\{ \begin{array}{l}
\phi_0=s^2\beta_1-\beta_0(st+t^2)\\
\phi_1=\beta_2(s^2+st)-\beta_1t^2\\
\phi=\beta_1s^2t-\beta_0st^2+\beta_2s^2t-\beta_1st^2
\end{array}\right. .$$

Let us consider our parameterization of $C$ which can be written as: 
$$
\left\lbrace \begin{array}{l}f_0\ =\ \alpha_1\beta_2-\alpha_2\beta_1 = s[s^2\beta_1-\beta_0(st+t^2)] = s\phi_0 \\  f_1\ =\ \alpha_2\beta_0-\alpha_0\beta_2 = \beta_0t^3-\beta_2s^3\\ 
f_2\ =\ \alpha_0\beta_1-\alpha_1\beta_0 =t[\beta_2(s^2+st)-\beta_1t^2] = t\phi_1\end{array}\right. .
$$
Since 
$$
t\phi_0+s\phi_1= \beta_1s^2t-\beta_0st^2+\beta_2s^2t-\beta_1st^2-(\beta_0t^3-\beta_2s^3)=\phi - f_1,
$$
we get:
$$
\left\lbrace \begin{array}{l}f_0\ =\  s\phi_0 \\  f_1\ =\ \phi-(t\phi_0+s\phi_1) \\ 
f_2\ =\  t\phi_1\end{array}\right.  .
$$
 
\medskip
Now we can consider the curve $D\subset \PP^4$ given by the parameterization:
$$
\left\lbrace \begin{array}{l}x_0\ =\  s\phi_0 \\  x_1\ =\ -t\phi_0 \\ 
x_2\ =\ -s\phi_1 \\ x_3\ =\ t\phi_1\\ x_4\ =\ \phi\end{array}\right. .
$$
 We get that $C$ can be viewed as the projection of $D$ on the plane $H = \{x_2=x_4=0\}$ from the points $P_1=(0,1,-1,0,0)$ and $P_2=(0,1,0,0,-1)$.

\medskip
We want to show that $D$ is smooth. It not hard (e.g. by using CoCoA \cite{CoCoA-5}) to see that $D$ is contained in three quadric hypersurfaces $\{ x_0x_3 -  x_1x_2 ,
 x_2x_3 - x_2x_4 + x_3x_4+x_0x_3, 
  x_0x_1  + x_0x_4 - x_1x_4-x_1x_3 \}$, which generate the ideal of a smooth rational normal surface $S$ such that $P_1,P_2\notin S$.
\medskip
The structure of $S$ is well known; Pic$(S)$ is generated by $C_0,f$, where $f$ is a fiber and $C_0$ is a  curve with $C_0^2=-e=-1$, $C_0\cdot f=1$ and $f^2=0$. 

It is not hard to check that the lines of the ruling of $S$ are contained in the planes of the ruling of the quadric cone $\mathcal{Q} = \{ x_0x_3 -  x_1x_2=0 \}$, and such planes are unisecant (outside its vertex) to $D$. Hence $D$ has to be of type $(C_0+yf)$ on $S$ and so it is smooth.

We also have that the hyperplane section of $S$ is $C_0+2f$, and, since $D$ has degree $d$, we have:
$$
D\cdot H = -1+y+2= 1+y=d 
$$
which yields $D=C_0+(d-1)f$. 

Now we sketch the construction in the Ascenzi case (similar to what we just did). If $C$ has a point of multiplicity $d-3$, we can suppose 
$$
\alpha_0=s^2(as+bt); \ \alpha_1=t^2(cs+dt), \ \alpha_2=0,$$ 
for some $a,b,c,d\in K$,  with $ a\neq 0,\ d\neq 0$ (since $\alpha_0,\alpha_1$ do not have common factors).

Working as in the previous case, we can construct the curve $D\in \PP^4$ given by the parameterization:
$$ 
\left\lbrace \begin{array}{l}x_0\ =\  d\beta_2 t^3
 \\  x_1\ =\ a\beta_2 s^3 \\ 
x_2\ =\ (a'+b)\beta_2s^2t
 \\ x_3\ =\  (c+d')\beta_2st^2 \\
  x_4\ =\phi\end{array}\right. 
$$
where $\phi=f_2-(a'\beta_2s^2t+d'\beta_2st^2)$, $a'=a$ if $b=0$ and $a'=0$ if $b\neq 0$,  $d'=d$ if $c=0$ and $d'=0$ if $c\neq 0$.
It is easy to check that the curve $D$ lies on a rational normal scroll $S$
which is a cone, with vertex $V=(1,0,0,0,0)$ and $D$ has a $(d-3)$-ple point at $V$, given by the $d-3$ zeroes of $\beta_2$. {Hence also in this case} $C$ is a projection of $D$ as requested, and we are done.
\end{proof}
 
 \bigskip
This way to directly construct the curve $D$ on a scroll is not so feasible when $d\geq 4$; we will follow another attack in order to generalize this idea in the next section.

 \bigskip
 \section{The main theorem}\label{MainThm}
 
In this section we state and prove our main result:

\begin{thm}\label{ex-conjecture} Let $C\subset \mathbb{P}^2$ be a rational curve of degree $d$. The splitting type of $C$ is $(k,d-k)$, $2\leq 2k\leq d$, if and only if there is a degree $d$ rational curve, $\mathcal{D}\subset \mathbb{P}^{k+1}$  which is contained in a rational normal surface (scroll) $S_{h,k-h}\subseteq \PP^{k+1}$, with $2h\leq k$, the curve $C$ is the projection of $\mathcal{D}$ from $(k-1)$ points outside $S_{h,k-h}$ and either 

i) $h>0$, $S_{h,k-h}$ and $\mathcal{D}$ are smooth (and $C$ is not Ascenzi), or 

ii)  $S_{0,k}$ is a cone , $\mathcal{D}$ passes through the vertex $V$ of the cone with multiplicity $d-k$ and $V$ is its only singular point {(this covers also the case $S_{0,1}=\PP^2$, for which $C=\mathcal{D}$). In this case $C$ is Ascenzi.} 

Moreover $\mathcal{D}$ is uniscant to the fibers of $S_{h,k-h}$ and $h$ is the degree of a syzygy of minimal degree for a syzygy of minimal degree $k$ of $C$. 
\end{thm}

\bigskip
\begin{proof}
{Notice that for $k=1,2$, the theorem is known, proved by Ascenzi in \cite{refAs1}, while for $k=3$ it can be deduced from Proposition 2.2 above.}

Let $C$ be parameterized by $(f_0,f_1,f_2)\in K[s,t]_d$, then the ideal $(g_0,g_1,g_2) = (x_2f_0-x_0f_2,-x_1f_2+x_2f_1,x_0f_1-x_1f_0)\subset K[x_0,x_1,x_2,s,t]$ defines a curve $D$ in $\PP^2\times \PP^1$, which is the graph of the map $\PP^1 \rightarrow \PP^2$ (the projection of $D$ on the $\PP^1$ factor is 1:1 and surjective while the projection on the $\PP^2$ factor is $C$). 

 Let $I_D\subset K[x_0,x_1,x_2,s,t]$ be the ideal of $D$, we have that a polynomial $A_0x_0+A_1x_1+A_2x_2\in K[x_0,x_1,x_2,s,t]_{1,n}$ vanishes on $D$ if and only if $A_0f_0+A_1f_1+A_2f_2$ is identically zero in $ K[s,t]$ (in the language of \cite{refCSC}, this is a {\it moving line of degree $n$ for $C$}), i.e. the spaces $(I_D)_{1,n}$ correspond to the syzygies of degree $n$ of the ideal $(f_0,f_1,f_2)\subset K[s,t]$.  
 
 Hence we have that the splitting type of $C$ is $(k,d-k)$, $2k\leq d$, if and only if $k$ is the first integer for which $(I_D)_{1,k}\neq \{0\}$, and we can distinguish two cases: either $2k<d$ or $2k=d$.  From what we know about the syzygies of $(f_0,f_1,f_2)$,  if $2k<d$, there is a unique (up to multiplication by constants) form $p\in (I_D)_{1,k}$, namely $p = \alpha_0x_0+\alpha_1x_1+\alpha_2x_2$; all elements of $(I_D)_{1,n}$, for $n=k+1,\ldots ,d-k-1$ are of type $\lambda p$, with $\lambda\in K[s,t]_{n-k}$, and there is an element $q\in (I_D)_{1,d-k}-(p)_{1,d-k}$ such that  $(I_D)_{1,d-k}=(p,q)_{1,d-k}$ (of course we can choose $q$ modulo elements of $(p)_{1,{d-k}}$). If $2k=d$ there are two independent forms $p, p_1\in (I_D)_{1,k}$ which generate $\bigoplus_l(I_D)_{1,l}$.

\medskip

Now let us consider the moving line of degree $k$ for $C$, given by $p = \alpha_0x_0+\alpha_1x_1+\alpha_2x_2\in K[x_0,x_1,x_2,s,t]$; $p=0$ is the equation of a rational scroll   $S\subset \PP^2\times \PP^1$: for every $(s,t)\in \PP^1$ we have a line in the corresponding $\PP^2$-fiber. 

\bigskip
 \begin{lem} The scroll $S$ and the curve $D$ are smooth.
 \end{lem}
\begin{proof} 
Let us consider the partial derivatives of $p$:
$$
\left(\dfrac{\partial p}{\partial x_0}\ ,\ \dfrac{\partial  p}{\partial x_1}\ ,\ \dfrac{\partial p}{\partial x_2}\ ;\ \dfrac{\partial  p}{\partial s}\ ,\ \dfrac{\partial  p}{\partial t}\right)\ = \ \left( \alpha_0, \alpha_1, \alpha_2; p_s, p_t\right).
$$
Since the $\alpha_i$'s cannot have common factors (they represent a minimal syzygy), we get that in the above vector the first part in the $\alpha$'s is never zero at any point of $S$.  From this it is not hard to check that we can find affine charts where the vector of local derivatives is always non-zero, hence $S$ is smooth. 

Then, since $D$ is a curve on $S$ which is unisecant for its fibers, $D$ is smooth.
  
\end{proof}

We want to show that the scroll $S_{h,k-h}$ which appears in our statement can be obtained from the scroll $S\subset \PP^2\times\PP^1$ that we have  just defined as $\{p=0\}$, as the image of $S$ under the map given by a suitable linear system on it. 

\medskip
The idea is to consider the ideal $(\alpha_0,\alpha_1,\alpha_2)\subset K[s,t]$, which in turns gives a finite map $\mathbb{P}^1\rightarrow \mathbb{P}^2$, maybe not generically one to one.
Let us consider the splitting type of $(\alpha_0,\alpha_1,\alpha_2)$, say  $(h,k-h)$, $2h\leq k$, where $h$ is the degree of the minimal syzygy $(\gamma_0,\gamma_1,\gamma_2)$; we will get as before a curve $\Gamma\in \PP^2\times\PP^1$ and an ideal $(p_1,q_1)$, where
$$
q_1 = \gamma_0x_0+\gamma_1x_1+\gamma_2x_2, \quad q_2 = \delta_0x_0+\delta_1x_1+\delta_2x_2,  
$$
and
$$\left(
\begin{matrix}
\gamma_0 & \gamma_1 & \gamma_2    \\
\delta _0 & \delta _1 & \delta _2  
\end{matrix}
\right)$$
is the syzygy matrix of $(\alpha_0,\alpha_1,\alpha_2)$. Hence the $\delta_i$'s give a primitive syzygy of degree $k-h$ for $(\alpha_0,\alpha_1,\alpha_2)$. 

\begin{lem}   When we consider $S$ as a Hirzebruch surface $S=\mathbb{F}_e$, with $C_0^2=-e$, we have  $e=k-2h$ and $C_0$ is the graph of the curve in $\PP^2$ parameterized by $(\gamma_0,\gamma_1,\gamma_2)$. Moreover, we have that the graph of the curve parameterized by $(\delta_0,\delta_1,\delta_2)$ is $C_1\sim C_0+ef$ (on $S$), and the curve $D$ is, as a divisor on $S$, $D\sim C_0+(d-h)f$.
\end{lem}
\medskip
\begin{proof}[Proof of the Lemma]  The unisecant divisors on $S$ are in 1-1 correspondence with syzygies of $(\alpha_0,\alpha_1,\alpha_2)$ (this is quite obvious, see also \cite{refCSC}), and for any such syzygy $(A_0,A_1,A_2)$, with $A_i\in K[s,t]_n$, $n\geq h$, we have:
$$
(A_0,A_1,A_2) = \lambda (\gamma_0,\gamma_1,\gamma_2) + \mu (\delta_0,\delta_1,\delta_2)
$$
where $\lambda\in K[s,t]_{n-h}$, $\mu \in K[s,t]_{n-k+h}$. For any such syzygy, the graph of the curve parameterized by $(A_0,A_1,A_2)$ in $\PP^2$  is a divisor on $S$ which is unisecant to its fibers (and is cut out in $\PP^2\times\PP^1$ by $x_0A_1-x_1A_0, x_0A_2-x_2A_0,x_1A_2-x_2A_1$). 

When $h\neq k/2$ we consider the unique effective divisor $C_0$ which is associated to the minimal syzygy  $(\gamma_0,\gamma_1,\gamma_2)$, we have that for any $n>h$, we get divisors which are linearly equivalent to $C_0+(n-h)f$ (where $f$ is a fiber); hence $C_0$ is the only unisecant divisor class which has $H^0(\mathcal{O}_{S}(C_0))=1$ and which does not contain divisors containing fibers, i.e. we have $S\cong \mathbb{F}_e$, as a Hirzebruch surface, with $C_0^2=-e$. Moreover, we get that all the effective divisors given by syzygies of degree $h+1\leq n \leq k-h-1$ are not irreducible and are given by $C_0$ plus $n-h$ fibers (the zeroes of $\lambda$); the first value of $n$ for which we have a class of divisors  $C_0+(n-h)f$ whose generic element is an irreducible one, is $n=k-h$, for which we have $C_1$ given by the graph of  $(\delta_0,\delta_1,\delta_2)$. Notice that the curve $C_1$ is actually linearly equivalent to $C_0+ef$, since they differ by a principal divisor  that, locally, is of the form $\frac{x_i\delta_j-x_j\delta_i}{\lambda(x_u\gamma_v-x_v\gamma_u)}$.
Hence, for what we know on rational ruled surfaces (e.g.  {see \cite{segre1884sulle}, or}  \cite[Ch. V, Thm 2.17]{refHa}), $C_1\sim C_0+(k-2h)f$, $e=k-2h$, and $C_0\cdot C_1=0$. 

If we consider the curve $D$, this is the graph of $(f_0,f_1,f_2)$, which is a syzygy of degree $d$ for $(\alpha_0,\alpha_1,\alpha_2)$, hence we have $D\sim C_0+(d-h)f$. 

When $h=k/2$ the choice of the $C_0$ is not unique, in fact we can choose as $(\gamma_0,\gamma_1,\gamma_2)$ and $(\delta_0,\delta_1,\delta_2)$ any basis of the module of syzygies of degree $h$, and the we can proceed as in the previous case (here $C_0 \sim C_1$).
\end{proof}

\medskip
Once the lemma is true, we can consider the linear system $H\sim C_0+(k-h)f$ on $S$; this is very ample if $h>0$, and it will give an embedding of $S$ as a surface of degree $H^2 = C_0^2+2(k-h)=2h-k+2k-2h=k$ in $\PP^{k+1}$ (i.e. as a rational normal scroll); the embedded scroll is of type $S_{h,k-h}$, i.e. it is given by the lines joining two common parameterizations for two rational normal curves, which are the embedding of $C_0$ and $C_1$, of degrees $h$, $k-h$, in a $\PP^h$, $\PP^{k-h}$, respectively. 

In the case $h=0$, the linear system is not very ample, but it still gives a map to $\PP^{k+1}$; it only contracts $C_0$ to a point, and we still get a rational normal scroll $S_{0,k}$, which is a cone. 

Notice also that the case $h=0$ corresponds to the case when $C$ is Ascenzi, since $(\gamma_0,\gamma_1,\gamma_2)$  has degree 0, hence, modulo a projective change of coordinates, we can assume $\alpha_2=0$, so  $(\alpha_0,\alpha_1,0)$ has a syzygy of degree 0, namely $(0,0,1)$; in this case, $C_0=\{x_0=x_1=0\}$, $(\delta_0,\delta_1,\delta_2)=(-\alpha_1,\alpha_0,0)$ and the scroll we get is the cone $S_{0,k}$. The curve $\mathcal{D}$, since $D\cdot C_0 = -e +d = d-k$, is singular in the vertex of the cone, with multiplicity $d-k$. Moreover, in the case $k=1$, $S_{0,1}$ is $\PP^2$ itself.
\medskip

Now we have to check that the curve $\mathcal{D}\subset S_{h,k-h}\subset \PP^{k+1}$ that we have constructed really projects to $C\subset \PP^2$.

The embedding of $S$ is defined by the linear system $\vert H\vert = |C_0+(k-h)f|$, which corresponds to syzygies of $(\alpha_0,\alpha_1,\alpha_2)$ of degree $k$; if $(A_{0,i},A_{1,i},A_{2,i})$ , $i=0,\ldots , k+1$ is a base for the module of syzygies of degree $k$, to any of them it is associated a divisor on $S$ which is the graph of the curve in $\PP^2$ parameterized by $(A_{0,i},A_{1,i},A_{2,i})$; such divisor can be locally defined by one of the three  functions 
{$$x_2A_{1,i}-x_1A_{2,i},\ x_0A_{2,i}-x_2A_{0,i},\ x_0A_{1,i}-x_1A_{0,i};$$ }
hence the map $\phi : S \rightarrow \PP^{k+1}$ given by $|H|$ is defined, locally, by such equations.

Consider  the three trivial syzygies: $(0,\alpha_2,-\alpha_1)$, $(\alpha_2,0,-\alpha_0)$, $(\alpha_1,-\alpha_0,0)$ as the last three syzygies of our base; they define a plane $\Pi \cong \PP^2$ inside our $\PP^{k+1}$. We want to check that the projection of $\PP^{k+1}$ on $\Pi$ maps $\mathcal{D}$ onto our starting curve $C$.

 The three local equations we saw above, for each of these syzygies, can be taken to be (by considering that on $S$ we have $x_0\alpha_0+x_1\alpha_1+x_2\alpha_2=0$):
$$
x_0\alpha_2,\  x_0\alpha_1,\ -x_2\alpha_2-x_1\alpha_1=x_0\alpha_0; \quad  x_1\alpha_2,\ -x_2\alpha_2-x_0\alpha_0=x_1\alpha_1,\   x_1\alpha_0;$$ $$-x_1\alpha_1-x_0\alpha_0= x_2\alpha_2,\  x_2\alpha_1,\  x_2\alpha_0.
$$
Since $(\alpha_0,\alpha_1,\alpha_2)$ cannot have common factors, at any point $(x_0,x_1,x_2;s,t)\in S$, there is at least one $\alpha_i(s,t)\neq 0$; so its image under the projection can be written $(\alpha_i(s,t)x_0,\alpha_i(s,t)x_1,\alpha_i(s,t)x_2)$. Hence the projection map from $\phi(S)=S_{h.k-h}$ to $\Pi$ can actually be written as:
$\phi'(x_0,x_1,x_2;s,t) \mapsto (x_0,x_1,x_2)$, so the image of the curve $\mathcal{D}$ is actually our starting curve $C$.

Notice that, by Proposition \ref{Projection}, $C$ cannot be the projection of a curve $D$ on a rational normal scroll in $\mathbb{P}^{k'+1}$ with $k'< k$.
\end{proof}

{
\begin{rem}\label{geometric} We would like to notice that the construction developed in the proof of Theorem \ref{ex-conjecture} can be extended in order to see our curve $\mathcal{D}\subset \PP^{k+1}$ as the projection of a rational normal curve.  \\
If we consider, on $S$, the linear system $\mathcal{L}$ which is given by the divisors in $|C_0+(d-h)f|$ which have $(d-k+1)$ fixed points $P_1, \ldots , P_{d-k}, P_{d-k+1}$, where $P_1$ and $P_{d-k+1}$ are on the same line and $P_1, \ldots , P_{d-k}$ are on $C_0\cap D$, we can embed $S$ into $\mathbb{P}^d$ as a rational normal scroll $S_{h,d-h-1}$ and the image of $D$ will be a rational normal curve. By projecting $S_{h,d-h-1}$ from $d-k-1$ other points $P_{d-k+2}, \ldots , P_{2d-k}$, where $P_i$ and $P_{d-k+i}$ are on the same line, we will get the scroll $S_{h,k-h}\subset \PP^{k+1}$ of our previous construction.
\end{rem}

\section{Open problems}\label{OpenPrb}

In this section we give a list of open problems.

\subsection{The center of projection}

As outlined in Remark \ref{geometric} it would be interesting to understand the relation between the splitting type of the plane curve $C$ and the geometry of the linear space $\Pi$, with  $\Pi\simeq \PP^{d-3}$ which is the center of  the projection that sends the rational normal curve $C_d\subset \PP^d$ to $C\subset \PP^2$. What we know, by  Remark \ref{geometric}, is that  there is a rational normal scroll containing  $C_d$  and intersecting  $\Pi$ in $d-k-1$ points.

\subsection{Singularities and splitting type}

Consider linear systems  $\mathcal{L}(d;m_1,\ldots ,m_s)$ of plane curves of degree $d$ and multiplicities at least $m_i$ at generic points $P_i\in \PP ^2$, $i=1,\ldots , s$, such that the generic curve $C$ in $\mathcal{L}(d;m_1,\ldots ,m_s)$ is rational. Find in which cases $(d;m_1,\ldots , m_s)$ determines the splitting type of $C$.  Of course a well known example when this happens is the Ascenzi case (e.g. a generic curve in $\mathcal{L}(d;d-1,1^{s})$, $s\leq 2d$, has splitting type $(1,d-1)$).   A different (non-Ascenzi) example of such a situation is given by $\mathcal{L}(8;3^7)$, i.e. curves of degree $d=8$ with seven points of multiplicities $m_i=3$. In this case the splitting type of a generic curve $C$ in this linear system has to be $(3,5)$ (e.g. see \cite{GHI2}).

Notice that just knowing the number and multiplicities of the singularities of $C$ in general is not enough to determine the splitting type of $C$. For example the generic projection in the plane of a degree 8 rational normal curve has only  nodes as singularities and splitting type $(4,4)$. On the other hand  the example given in the Introduction is of a plane curve degree 8 with only nodes as singularities but splitting type $(3,5)$, and, by Ascenzi's results, there is a plain curve $C$ which is a generic projection of a rational curve $D$ in $\mathbb{P}^3$ of type $(1,7)$ on a smooth quadric, which has splitting type $(2,6)$ and only nodes as singularities.  Clearly, in all cases, the nodes are not supported at generic points.\\

A deep analysis of the possibilities to describe singularities of certain rational plane curves via its parameterizations and its syzygies can be found in \cite{CoxetAl} 

\subsection{Curves in higher dimensional spaces} 
Generalize the type of construction that we give in Theorem \ref{ex-conjecture} for curves in $\PP^n$ with $n\geq 3$, i.e. give an explicit connection between the parameterization of $C$, its splitting type and the rational scrolls containing $C$ or a curve $D$ in higher space whose projection is $C$.  Notice that the idea of $\mu$-basis  is considered in \cite{refCSC} also in the case of curves in higher spaces.
}

\bigskip

\bibliographystyle{alpha}



\bibliography{Poncelet}

\def\biblio{

}
\end{document}